\documentclass[12pt]{amsart}
\usepackage{amsmath,amsfonts,amsthm,amsopn,cite,mathrsfs, amssymb}
\usepackage{epsfig,verbatim}
\usepackage{subfigure}

\newcommand{\fif}{if and only if}

\newtheorem{tm}{Theorem}
\newtheorem{lemma}[tm]{Lemma}

\newtheorem{cor}[tm]{Corollary}

\theoremstyle{definition}

\newcommand{\beqa}{\begin{eqnarray*}}
\newcommand{\eeqa}{\end{eqnarray*}}

\DeclareMathOperator*{\supp}{supp}

\newcommand{\field}[1]{\mathbb{#1}}
\newcommand{\bR}{\field{R}}        
\newcommand{\bN}{\field{N}}        
\newcommand{\bZ}{\field{Z}}        
\newcommand{\bC}{\field{C}}        
\def\cO{\mathcal{ O}}

\def\<{\left<}
\def\>{\right>}

\def\inv{^{-1}}

\def\mv1{M_v^1}

\newcommand{\sis}{shift-invariant}
\begin{document}
\begin{abstract}
We study the problem of recovering  a function of the form $f(x) = \sum
_{k\in \bZ } c_k e^{-(x-k)^2}$ from its phaseless samples $|f(\lambda
)|$ on some arbitrary countable set $\Lambda \subseteq \bR $. For
real-valued functions  this is possible up to a sign  for every separated set with Beurling density
$D^-(\Lambda ) >2$. This result is sharp. For complex-valued functions we find all possible
solutions with the same phaseless samples. 
\end{abstract}

\title{Phase-Retrieval in Shift-Invariant Spaces with Gaussian Generator}
\author{Karlheinz Gr\"ochenig}
\address{Faculty of Mathematics \\
University of Vienna \\
Oskar-Morgenstern-Platz 1 \\
A-1090 Vienna, Austria}
\email{karlheinz.groechenig@univie.ac.at}
\subjclass[2010]{94A12,42C15,49N30}
\date{}
\keywords{Phase-retrieval, shift-invariant space, Gaussian,
  Hadamard factorization}
\thanks{K.\ G.\ was
  supported in part by the  project P31887-N32  of the
Austrian Science Fund (FWF)} 
\maketitle


In many measurements only non-negative  amplitudes or intensities are
recorded rather than  the physical quantity  itself. The problem then
is to recover the underlying object (signal, image) from
these phaseless measurements, in other words, to retrieve the phase
from the  magnitude.
This is the problem of phase-retrieval. Particular problems concern
the recovery of a function $f$ from its phaseless  Fourier
measurements $|\hat{f}(\xi )|$ or the recovery of a function $f$ from
its phaseless (or unsigned) samples $|f(\lambda )|$.   For general
information about the phase-retrieval problem we refer to 
the surveys~\cite{SEC15,GKR19}.

The question of phase-retrieval is usually ill-posed or even
meaningless, unless one imposes  additional assumptions on the
function to be reconstructed, such as 
a particular  signal model.    
In this note we consider the problem of phase-retrieval in the 
particular class of  shift-invariant spaces generated by a
Gaussian. The prototype of a shift-invariant space is the Paley-Wiener
space of bandlimited functions $\{ f\in L^2(\bR ): \supp \, \hat{f}
\subseteq [-1/2, 1/2]\}$. More general shift-invariant spaces serve as
a substitute for bandlimited functions and have received wide
attention in approximation theory and sampling theory~\cite{AG00,BVR94}. 

Given a generator $g\in L^2(\bR )$,  $p\in [1,\infty]$, and a mesh
parameter or step size $\beta >0$, let
\begin{equation}
  \label{eq:1}
  V^p_\beta (g) = \big\{ f = \sum _{k\in \bZ } c_k g(\cdot - \beta k): c\in \ell ^p(\bZ
  )\big\}  \, .
\end{equation}

One of the versions of phase-retrieval in \sis\ spaces  is the recovery
of a function $f$ (up to a scalar) from its phaseless  samples
$|f(\lambda )|$ on some set $\Lambda $.  The question then is whether
the additional information  that   $f$ belongs to the
\sis\ space $V^p_\beta (g)$ suffices to determine $f$ up to a sign.  
For the prototype of a \sis\ space, namely the bandlimited functions,
this question was solved 
~\cite{Thak11}. Recently Q.\ Sun and his
collaborators have developed a general theory for phase-retrieval in
\sis\ spaces in a series of articles~\cite{CCS19,CCSW19,CS19a}. A typical result asserts that the samples of
$|f|$ on a sufficiently dense union of shifted lattices suffice to
recover real-valued functions in some $V^2(g)$. These papers also
cover some of the numerical aspects of
phase-retrieval. The problem of  phase-retrieval in \sis\ spaces from Fourier
measurements is studied in~\cite{SMS16}. 


We study the problem of  phase-retrieval in the \sis\
space generated by a Gaussian $\phi _\gamma (x) = e^{-\gamma x^2}$. 
Precisely,  we will assume that $f$ is a linear combination of shifts
of a Gaussian and belongs to the  \sis\ space
$$
V^\infty _\beta (\phi _\gamma ) = \big\{ f\in L^\infty (\bR ): f(x) =
\sum _{k\in \bZ}
c_k e^{-\gamma (x-\beta k)^2}  \quad c \in \ell ^\infty (\bZ
)\big\} \, .
$$
We will allow samples from an  arbitrary separated set $\Lambda \subset
\bR $ and measure the 
density of $\Lambda $  with the standard notion of the lower  Beurling density   defined as
\begin{equation}
  \label{eq:2}
D^{-}(\Lambda) := \liminf_{r \rightarrow \infty}
\inf_{x \in \bR} \frac{\# \Lambda \cap [x-r,x+r] }{2r} \, .
\end{equation}

 We will  first consider  consider real-valued functions and 
unsigned samples $|f(\lambda )|$. For this case   
we  prove the following uniqueness  theorem.
\begin{tm} \label{main}
  Assume that $\Lambda \subseteq \bR $ is separated and that
  $D^-(\Lambda ) > 2\beta \inv $. Then phase-retrieval is possible on
  $\Lambda $ for all 
  real-valued functions in   $V_\beta ^\infty (\phi
  _\gamma )$. This means that an unknown $f\in V^\infty _\beta (\phi
  _\gamma )$  is uniquely determined by its phaseless samples
  $\{  |f(\lambda )| : \lambda \in \Lambda \}$ up to a sign. 
\end{tm}

Theorem~\ref{main} says that for real-valued  $f, g \in
V^\infty _\beta (\phi _\gamma )$ satisfying  $|f(\lambda ) | = |g(\lambda )|$
for all $\lambda \in \Lambda $ we have either  $g = f$ or $g=-f$, thus the only
ambiguity is  the (global)
sign of $f$. Our proof yields a reconstruction
procedure (see Section~\ref{complex}), but it is well-known that the phase-retrieval
problem in infinite dimensions is necessarily  ill-posed~\cite{AGr17,CCD16,GKR19}.

In the second part we will study complex-values functions in $V^\infty
_\beta (\phi _\gamma )$. In this case there is no uniqueness, instead
we will classify all functions $g\in V^\infty
_\beta (\phi _\gamma )$ that satisfy $|g(\lambda )| = |f(\lambda )|$
on $\Lambda $. 

\vspace{3mm}

Before proceeding, let us discuss some of the 
fine points of Theorem~\ref{main}.

(i)   Remarkably, the uniqueness  holds even for bounded
coefficients, and not just for  square-summable  coefficients.  We can therefore not rely on the Hilbert space theory
for phase-retrieval, but technically we need to rely on the Banach
space set-up of Alaifari and Grohs~\cite{AGr17}.

(ii) The density condition is sharp.  For uniform density $D(\Lambda ) <
2\beta \inv $ one can produce  essentially different real-valued
functions $f,g$ with the same phaseless 
samples $|g(\lambda ) | = |f(\lambda )|, \lambda \in \Lambda $. See
below.

(iii) A similar  statement  for bandlimited functions was proved  by 
Thakur~\cite{Thak11} and revisited in~\cite{AGr17}.  Both \cite{Thak11} and Theorem~\ref{main} support
the intuition that the recovery of phase from magnitude requires twice
as many samples as the recovery from the samples $f(\lambda )$. This
is well known for finite-dimensional frames, but in infinite
dimensions it  is  more subtle
to formulate and prove.  To our knowledge Theorem~\ref{main} is one of
only two models 
for which phase-retrieval is possible with a sharp density
condition. The investigations in~\cite{CCS19,CCSW19,CS19a} require a
much higher sampling rate or deal with conditions under which
phase-retrieval is not even possible. 

In Section~\ref{real} we collect several statements about the \sis\
space $V^\infty _\beta (\phi _\gamma )$ and then prove
Theorem~\ref{main}. In Section~\ref{complex} we study phase-retrieval
for complex-valued functions. There  we use a factorization
of period entire functions whose proof is postponed to
Section~\ref{hadaa}.  

\section{Phase-Retrieval for Real-Valued Functions}
\label{real}

We set up the proof of Theorem~\ref{main}. 
To avoid unnecessary parameters, we 
  set $\beta =1$,  without loss of generality. 
 This is possible because $f(x) = \sum _{k}
c_k e^{-\gamma (x-\beta k)^2} = \sum _{k}
c_k e^{-\gamma \beta ^2 (x/\beta - k)^2} $. This means that  $f \in
V^\infty _\beta (\phi _\gamma )$, \fif\ $f ( \beta \cdot ) \in
V^\infty _1 (\phi _{\beta ^2\gamma })$. Thus   phase-retrieval on
$\Lambda $ is possible  for $V_1^\infty(\phi _{\beta ^2\gamma})$, \fif\
phase-retrieval on $\beta \Lambda $ is possible  for $V_\beta^\infty (\phi _{\gamma})$. We
 note that $D^-(\beta \Lambda ) = \beta \inv D^-(\Lambda )$. Thus it
 suffices to prove Theorem~\ref{main} for $\beta =1$. \\

Our first use of complex variable methods is the following lemma  for Fourier series.

 \begin{lemma} \label{comp}
   (i) Let $d$ be a sequence with Gaussian decay with decay parameter
   $\gamma >0$, i.e., $d_k = c_k
   e^{-\gamma k^2}$ for some $c\in \ell ^\infty (\bZ )$. Then the Fourier series $\hat{d}(\xi ) = \sum _{k\in \bZ } d_k
   e^{2\pi i k\xi }$ can be extended to an entire function $D(z) = \hat{d}(\xi +
   iy)$ with growth  of order $2$, precisely,  $|D(\xi +iy)| = \cO ( e^{\pi ^2
     y^2 / \gamma }) $.

   (ii) Conversely, if $D(z)$ is a periodic  entire function  $D(z+k)
   = D(z)$ for all $z\in \bC $ and $k\in \bZ $ with growth $|D(\xi
   +iy)| = \cO ( e^{\pi ^2
     y^2 / \gamma })$, then the Fourier series of $D(\xi )$ has
   coefficients of Gaussian decay $d_k = c_k e^{-\gamma k^2}$ for some
   $c\in \ell ^\infty (\bZ )$. 
 \end{lemma}
 \begin{proof}
   For completeness we give  the elementary proof.

   (i) Writing $z= \xi +iy$ and $d_k = c_k e^{-\gamma k^2}$,  we have
   \begin{align*}
     \hat{d}(\xi +iy) &= \sum_{k\in \bZ  } c_k e^{-\gamma k^2} e^{2\pi
                        i k (\xi +iy)} \\
     & \leq \|c\|_\infty \,  \sum_{k\in \bZ  }  e^{-\gamma k^2} e^{2\pi
       k |y|} \\
     &= \|c\|_\infty \,  \Big( \sum_{k\in \bZ  }  e^{-\gamma (k - \pi
       |y|/\gamma )^2} \Big) e^{\pi ^2 y^2 / \gamma } = \cO ( e^{\pi
       ^2 y^2/\gamma } )\, .
   \end{align*}
   Clearly $\hat{d}$ is entire.
   
   (ii) If $D$ is entire and periodic, then  
   $$
D(z) = \sum _{k\in \bZ } d_k e^{2\pi i k z } = \sum _{k\in
     \bZ } d_k e^{-2\pi ky } e^{2\pi i k \xi } \, .
   $$
with uniform convergence on compact sets and exponentially decaying
coefficients. See, e.g., \cite[Thm.~3.10.3]{simon2}. 
Consequently the Fourier coefficients of $\xi \mapsto D(\xi +iy)$ are $d_k
e^{-2\pi k y}$ and  satisfy
$$
|d_k |e^{-2\pi k y} \leq \int _{0}^1 |D(\xi +iy)| \, d\xi \leq
C e^{\pi ^2 y^2 /\gamma } \, ,
$$
for \emph{all} $k$ and $y$, where the last inequality follows from the assumption. 
 Setting $y= -\frac{\gamma k}{\pi} $ yields the desired estimate
 $$
|d_k| e^{\gamma k^2} 
\leq C \, .  
 $$
 \end{proof}

 The analysis of phase-retrieval in $V^\infty _1 (\phi _\gamma )$
 involves  several steps.

 \vspace{3mm}
\noindent \textbf{Step 1.} \emph{From $f\in V^\infty _1(\phi _\gamma )$ to
   $|f|^2$.} We start with a simple algebraic observation. 

 \begin{lemma} \label{l1}
   If $f\in V^\infty _1(\phi _\gamma )$, then $|f|^2 \in
   V_{1/2}^\infty (\phi _{2\gamma })$. 
 \end{lemma}
 \begin{proof} 
 We write  $f(x)
= \sum _{k\in \bZ } c_k e^{- \gamma (x - k)^2}$ with $c\in \ell ^\infty (\bZ
)$. Since
$$
(x-k)^2+(x-l)^2 = 2\Big(x-\frac{k+l}{2}\Big)^2 - \frac{(k+l)^2}{2}
+k^2+l^2 \, ,
$$
we find that
\begin{align}
  |f(x)|^2 &= \sum _{k,l\in \bZ } c_k \overline{c_l} e^{-\gamma (k^2 +
             l^2)}
             e^{\gamma (k+l)^2/2} \, e^{-2\gamma
             \big(x-\frac{k+l}{2}\big)^2 } \notag \\
  &= \sum _{n\in \bZ  } \Big( \sum _{k\in \bZ } c_k \overline{c_{n-k}}
    e^{-\gamma k^2}
 e^{-\gamma (n-k)^2} \Big) e^{\gamma n^2/2} \, e^{-2\gamma (x-n/2)^2}
    \, .
      \label{eq:3}
\end{align}
This calculation shows  that the function $|f|^2$
belongs to a different shift-invariant space generated by $\phi
_{2\gamma}$ with step size $1/2$. Set
\begin{align}
  d_k &= c_k e^{-\gamma k^2} \, , \label{eq:16} \\
  r_n &= \sum _{k\in \bZ } c_k \overline{c_{n-k}}
        e^{-\gamma k^2} e^{-\gamma (n-k)^2} \, ,  \\
  \widetilde{r_n} & = r_n e^{\gamma n^2/2} \, . \label{eq:14}
\end{align}
 From these  definitions we see that
\begin{align}
  r&= d \ast \bar{d}  \qquad \text{ and }  \label{eq:15}  \\
  |f(x)|^2 &= \sum _{n\in \bZ } \widetilde{r_n} e^{-2\gamma (x-n/2)^2} \, .
\end{align}
If   $c\in \ell
^\infty(\bZ )$,  then $\tilde{r} \in \ell ^\infty (\bZ )$, because  
\begin{align*}
  |\widetilde{r_n}| &= \big| \sum _{k\in \bZ } c_k \overline{c_{n-k}}
                      e^{-\gamma k^2} e^{-\gamma (n-k)^2}\big|  \, e^{\gamma n^2/2} \\
  &\leq \|c\|_\infty ^2 \sum _{k\in \bZ } e^{ -\gamma (k^2 +
    (n-k)^2 - n^2/2) }  \\
  &= \|c\|_\infty ^2 \sum _{k\in \bZ } e^{-\gamma (n-2k)^2/2  } \, . 
\end{align*}
Setting   $C= \max (\sum _{k\in \bZ
} e^{-2\gamma k^2  }  , \sum _{k\in \bZ } e^{-\gamma (1-2k)^2/2  })$, we
have shown that
\begin{equation}
  \label{eq:12}
  \|\tilde{r}\|_\infty \leq C \|c\|_\infty ^2 \, .
\end{equation}
As a consequence $|f|^2
\in V^\infty _{1/2}(\phi _{2\gamma })$, as claimed. 
 \end{proof}


\noindent \textbf{Step 2.} \emph{A sharp sampling theorem.}  Our main tool is
the following  sampling theorem for
shift-invariant spaces with Gaussian generator  from \cite[Thm.~4.4]{GRS18}.
\begin{tm} \label{t2}
  Let $\gamma >0$ and $1\leq p\leq \infty $ and $\Lambda \subseteq \bR
  $ be separated.  If
$D^-(\Lambda ) > \beta \inv  $, then for some constants $A,B>0$
depending on $\Lambda $ and $p$ 
\begin{equation}
  \label{eq:ll1}
A \|f\|_p^p \leq \sum _{\lambda \in \Lambda }  |f(\lambda )|^p \leq B
\|f\|_p^p  \qquad \forall f\in V_{\beta }^p (\phi
_\gamma ) \, .  
\end{equation}
\end{tm}
Conversely, if \eqref{eq:ll1} holds, then $D^-(\Lambda ) \geq \beta
\inv $. This was already proved in \cite{AG00} and shows that
Theorem~\ref{t2} is optimal.

Applying this theorem to the function $|f|^2 \in V_{1/2}^\infty
(\phi _{2\gamma })$, we obtain the following consequence.
\begin{cor}
  Let $\gamma >0$ and $\Lambda \subseteq \bR $ be separated
  with density $D^-(\Lambda )>2$. If $f\in V_1^\infty(\phi _\gamma )$ and
  thus   $|f(x)|^2 = \sum _{n\in \bZ } \widetilde{r_n}
  e^{-2\gamma(x-n/2)}\in V_{1/2}^\infty (\phi _{2\gamma})$, then there
  exist constants $A,B>0$ such that 
  \begin{equation}
    \label{eq:13}
   A  \sup _{\lambda \in \Lambda } |f(\lambda )|^2 \leq  \|\tilde{ r}\|_\infty  \leq 
    B  \sup _{\lambda \in \Lambda } |f(\lambda )|^2   \, .
  \end{equation}
  Thus the coefficients $\tilde{r}$ of $|f|^2$ are uniquely and stably
  determined by the phaseless samples of $f$ on $\Lambda $. 
\end{cor}
Note that in  \eqref{eq:13} we have used the norm equivalence $\sup _{x\in \bR }
|f(x)|^2 \asymp \sup _{n\in \bZ } |\widetilde{r_n}|$. 


\vspace{3mm}

\noindent \textbf{Step 3.} \emph{A functional equation.}  The sampling
inequality \eqref{eq:13} allows us to 
recover the coefficients $\tilde{r}$ from the phaseless samples
$|f(\lambda )|^2$. Finally we have to recover the coefficients
$c$ and $d$ from the coefficients $\tilde r$ of $|f|^2$, or
equivalently from the  $r_n = \widetilde{r_n}
e^{-\gamma n^2/2}$.  

Let  $\hat{d}(\xi ) = \sum _{k\in \bZ } d_k e^{2\pi i k\xi
}$ be the 
Fourier series of $d$  and  $\hat{r}$ be the Fourier series of
$r$. Then equation~\eqref{eq:15} turns into 
\begin{equation}
  \label{eq:6}
  \hat{r}(\xi ) = \hat{d}(\xi ) \, \overline{\hat{d}(-\xi )} \, .
\end{equation}
Since $d_k = c_k e^{-\gamma k^2}$  has Gaussian
decay,  Lemma~\ref{comp} asserts that its  Fourier series can be extended to the  entire function
$$D(z) = \hat{d}(z) = 
\sum _{k\in \bZ } d_k e^{2\pi i kz } \, 
$$
with growth $|D(\xi +iy) |  = \cO ( e^{\pi ^2 y^2/\gamma })$. 
Likewise    $\overline{\hat{d}(-\xi )}$ extends to
the entire function  $ D^*(z) = \overline{D(-\bar{z})}=
\overline{\hat{d}(-\bar{z})}$,  and
  $\hat{r}$ extends to  $R(z)= \hat{r}(z)$.

  Consequently, we have to find the entire function $D$ that satisfies
  the 
  identity
  \begin{equation}
    \label{eq:7}
    D(z) \overline{D(-\bar{z})} = D(z) D^*(z) =  R(z) \, .
  \end{equation}
In other words, to every solution of the functional equation
\eqref{eq:7} corresponds a function $g$ such that $|g(\lambda )| =
|f(\lambda )|$ for $\lambda \in \Lambda $.

Assuming that $f$ is real-valued, we can now prove Theorem~\ref{main}
quickly.

\begin{proof}[Proof of Theorem~\ref{main}]
  Since   $f $ is real-valued by  assumption, its coefficients are
  also real-valued,   $c=\bar{c}$. This entails that  $ \overline{D(-\bar{z})} =
  \overline{\sum _{k\in \bZ} d_k e^{-2\pi i k(-\bar{z})} } = \sum
  _{k\in \bZ} d_k e^{2\pi i k z} = D(z) $ and \eqref{eq:7} becomes the
  identity
  \begin{equation}
    \label{eq:11}
    R(z) = D(z)^2 \, .
  \end{equation}
The uniqueness of $f$ up to a sign is now  immediate:  assume
that two entire (non-zero) functions  $D_1$ and $D_2$  satisfy $D_1^2 =
D_2^2 = R$. Then $(D_1-D_2)(D_1+D_2) \equiv 0$. Since the ring of   entire
functions does not have any zero divisors, we conclude that either
$D_1=D_2$ or $D_1=-D_2$ on $\bC $. 
Using formulas \eqref{eq:16} - \eqref{eq:14} we find that the
coefficients $c$ of $f$,  and thus $f$,   are uniquely determined by  the
phaseless samples $|f(\lambda )|, \lambda \in \Lambda $, up to a
sign.  Theorem~\ref{main} is therefore proved. 
 \end{proof}

Alternatively, one could prove Theorem~\ref{main} by verifying the
following criterium for phase-retrieval
~\cite{AGr17,CCD16}:  $\Lambda $ permits phase-retrieval, \fif\
$\Lambda $ satisfies the \emph{complement property}, i.e. if $S
\subseteq \Lambda $, then either $M_S = \{ f\in V^\infty _\beta (\phi
_\gamma ) : f(\lambda ) = 0, \, \forall \lambda \in S\} = \{0\}$ or
$M_{\Lambda \setminus S} = \{0\}$. A direct proof of the complement
property can  be based on Steps~1 and~2 and is then  similar to
the argument in~\cite[Thm.~2.5]{AGr17}.

\vspace{3mm}

Next we show that  the \emph{density condition is sharp}. 
Let $\Lambda  = (\lambda _j)_{j\in \bZ } \subseteq \bR$ be an increasing  sequence  with uniform  density
$D(\Lambda ) <2$.     This means that the upper Beurling density $D^{+}(\Lambda) := \limsup_{r \rightarrow \infty}
\sup_{x \in \bR} \frac{\# \Lambda \cap [x-r,x+r] }{2r}$ coincides with
the lower Beurling density and $D(\Lambda ) = D^+(\Lambda ) =
D^-(\Lambda ) <2$.

We will produce a counter-example to the complement property.  Set $S = \{ \lambda
_{2j}: j\in \bZ \}$ and $S^c = \{ \lambda
_{2j+1}: j\in \bZ \}$. Then $ S$ and $S^c$ are
disjoint and $D(S) = D(S^c) =
\tfrac{1}{2}\, D(\Lambda ) <1$. By the necessary density condition for
sampling in shift-invariant spaces, e.g.~\cite{AG00},  $S$ and $S^c$
cannot be sampling sets for $V^2_1(\phi _\gamma)$, but they are
interpolating sets by \cite[Thm.~1.3]{GRS18}. These facts imply  that $f\to \big(f(\lambda
)\big)_{\lambda \in S}$ is onto $\ell ^2(S)$  with non-trivial
kernel. Consequently,  there exist non-zero functions $f,g\in V^2_1(\phi _\gamma
)\subseteq V^\infty _1(\phi _\gamma)$ such that $f(\lambda ) = 0 $ for $\lambda \in S$ and
$g(\lambda ' ) = 0$ for $\lambda ' \in S^c$. By taking the 
real part of $f$ and $g$, we also may assume that
$f$ and $g$ are real-valued.  
Since for $\lambda \in \Lambda $ either $f(\lambda ) =0$ or $g(\lambda
) = 0$, we obtain 
\begin{align*}
|f(\lambda ) +  g(\lambda )|^2 &= |f(\lambda )|^2 + 2\,
                                   f(\lambda ) g(\lambda )
                                   + |g(\lambda )|^2 \\
&=   |f(\lambda ) -  g(\lambda )|^2 \, .
\end{align*}
By construction,  $f+g$ and $f-g$ are linearly independent, and thus
sign-retrieval is not unique.

\section{ Complex-Valued Functions} \label{complex}
Theorem~\ref{main} states that the unsigned  samples $\{|f(\lambda
)|\}$ determine a unique real-valued function $\pm f \in V^\infty _\beta (\phi _\gamma
)$. 
To treat  complex-valued functions in $V_\beta^\infty (\phi _\gamma )$, we use a factorization
of the entire periodic functions $D(z)$ and $R(z)$ that is
intermediate between  the
Hadamard factorization and a Blaschke product.

\begin{lemma} \label{hada}
  Let $D$ be an entire function of order $2$ satisfying the
  periodicity $D(z+l)$ $ = D(z)$ for all $z\in \bC $ and $l\in \bZ $. Let
  $m$ be the order of the zero at $z=0$, 
  $W_+$ be the zeros of $D$ in $\{x+iy \neq 0:  -1/2 < x \leq 1/2, y\geq 0\}$ and
  $W_-$ be the zeros of $D$ in $\{x+iy: -1/2 < x \leq
  1/2,  y<0\}$. Then $D$ possesses the factorization
  \begin{equation}
    \label{eq:oc1}
    D(z) = C \big(e^{2\pi i z} -1\big)^m \, e^{2\pi i r z}  \, \prod _{w\in W_+}
  \frac{e^{-2\pi i z} - e^{-2\pi i w}}{1-  e^{-2\pi i w} }\,  \prod _{w\in W_-}
  \frac{e^{2\pi i z} - e^{2\pi i w}}{1-  e^{2\pi i w} } 
    \end{equation}
    for some $r\in \bZ $ and $C\in \bC$. The product converges uniformly on compact
    sets. 
  \end{lemma}

We postpone the proof of this lemma to Section~\ref{hadaa} and first
discuss its application to  the phase-retrieval problem. 

 The factorization \eqref{eq:oc1} serves  to find all solutions $D$
to the equation $D(z) \overline{D(-\bar{z})} = R(z)$ in
\eqref{eq:7}. The  spirit of this argument is similar  to the 
analysis in~\cite{Aku56,JKP19,McDon04,Thak11}.

To avoid spelling out the product in \eqref{eq:oc1}, we use the
notation
\begin{equation}
  \label{eq:31}
   \Pi (W, m,r) = \big(e^{2\pi i z} -1\big)^m \, e^{2\pi i rz} \, \prod _{w\in W_+}
  \frac{e^{-2\pi i z} - e^{-2\pi i w}}{1-  e^{-2\pi i w} }\,  \prod _{w\in W_-}
  \frac{e^{2\pi i z} - e^{2\pi i w}}{1-  e^{2\pi i w} }  \, .
\end{equation}
Here $W$ is understood as a sequence $\{w_j : j\in \bN \}$ of zeros,
where elements may be repeated according to the (finite)  multiplicity of the
zero. Since $D$ is of order $2$, we know that $\sum _{j\in \bN } |w_j|^{-3}
<\infty $. See also  \eqref{eq:32} below.  With this understanding we obtain the following
convenient formulas for $\Pi (W)$. 

(i) Let $Jz = -\bar{z}$ be the reflection of $z\in \bC $ about
the imaginary axis and $D^*(z) = \overline{D(-\bar{z})}$.  For a
single factor $\phi (z) = \frac{e^{2\pi i z} - e^{2\pi i
    w}}{1 - e^{2\pi i w} }$ in \eqref{eq:oc1}, the involution is
$\phi ^*(z)= \overline{\phi (-\bar{z})} =  \frac{e^{2\pi i z} - e^{-2\pi i
  \bar{w}}}{1 - e^{-2\pi i \bar{w}} }$. Consequently
\begin{equation}
  \label{eq:33}
  \Pi (W,m,r)^* = \Pi (JW,m,r) \, . 
\end{equation}

(ii) Multiplicativity: Let  $V \, \dot{\cup } \, W$ be  the mixture  $\{v_1, w_1, v_2, w_2,
\dots \}$ of  two sequences. Note that if the sequences are disjoint, then $V \,
\dot{\cup } \, W = V\cup W$.  If $f=\Pi (V,m,r)$ and $h=\Pi (W,m',r')$, then
\begin{equation}
  \label{eq:34}
  f h = \Pi (V \dot{\cup } W,m+m',r+r') \, . 
\end{equation}

Now let $D = \Pi (W,m,r)$ be given. Then
$$
R = D D^* = \Pi (W,m,r) \Pi (JW,m,r) = \Pi (W\dot{ \cup } JW, 2m, 2r) \, .
$$
This implies that the order the zero at $0$ is even and that the
factor $e^{2\pi i z}$ occurs with an even power. Furthermore, since
$Jw=w$ for every $w = u+iw$ with $u= 0$ or $u= \pm 1/2$, the zeros on
the lines $i\bR $ and $\pm 1/2+i\bR $ also occur with even multiplicity in
$R$.

Now assume that $R= \Pi (Z,2m,2r)$ is given so that its
zero set  is symmetric $Z=JZ$ and that  the zeros on
the lines $i\bR $ and $1/2+i\bR $ have  even multiplicity.

Let $S_+ = \{ x+iy: 0<x<1/2\}$ and $S_0 = i\bR \cup (1/2+i\bR
)$.  Let $Z_0 $ be the zeros of $R$ 
in $Z\cap S_0$ counted  with half their multiplicity. 
In our notation this means that $Z_0 \dot{\cup } Z_0 = Z\cap S_0$. Next choose $V\subseteq Z \cap S_+$ arbitrary and set
\begin{align}
  W&= V \dot{\cup} J\Big((Z\setminus V) \cap S_+\Big) \dot{\cup} Z_0 \label{fact0}  \\
  D_V&= \Pi (W,m,r) \, .    \label{fact1}
\end{align}
Then
\begin{align*}
W\dot{\cup} JW &= \Big( V \cup J\big((Z\setminus V) \cap S_+\big) \cup Z_0 \Big)
           \cup  \Big( JV \cup \big((Z\setminus V) \cap S_+\big) \cup JZ_0 \Big) \\
  &= \Big( \big( V \cup  (Z\setminus V) \big)  \cap S_+ \Big)  \cup
    J \Big( \big( V
    \cup  (Z\setminus V) \big)  \cap S_+ \Big) \cup Z_0 \dot{\cup} JZ_0   \\
  &= (Z\cap S_+) \cup J(Z\cap S_+) \cup (Z\cap S_0) = Z \, , 
\end{align*}
and consequently
$$
D_V\, D_V^* = \Pi (W,m,r) \Pi (JW, m,r) = \Pi (W\dot{\cup} JW, 2m, 2r) = R \, .
$$
Thus every choice $V$ of zeros of $R$  in the strip $S_+$ with the
corresponding function $D_V$ yields a valid
factorization of $R$. Clearly, different zeros sets $V_1, V_2$  (counting multiplicities),
yield  $D_{V_1} \neq D_{V_2}$. 

Conversely, every factorization $DD^* = R$ arises in this way, because 
we can always write the zero set of $D$ as $W = (W\cap S_+)  \cup
(W\cap S_0) \cup (W\cap JS_+)$ and we may choose $V = W\cap S_+$ to
recover $W$ from $Z= W \cup JW$. 


To summarize, we state the following lemma.
\begin{lemma} \label{summ}
  Let $R=R^* = \Pi (Z,2m,2r)$ be a periodic entire function of order
  $2$ with all zeros on $i\bR \cup 1/2 + i\bR $ of even multiplicity. Then
  every solution of $DD^*=R$ is given by some $D_V$, as defined in
  \eqref{fact0} and \eqref{fact1}. 

  If $|R(\xi +iy)| = \cO (e^{2\pi ^2 y^2/\gamma })$, then $|D_V(\xi
  +iy)| = \cO(e^{\pi ^2 y^2/\gamma })$.
\end{lemma}
\begin{proof}
  Since $D_V$ is entire and periodic, its growth is $|D_V(\pm \xi +iy)| =
  \cO (e^{\psi (y)})$. Therefore $|D_V(\xi +iy) D_V(-\xi +iy)| = \cO
  (e^{2\psi (y)}) = \cO (e^{2\pi ^2 y^2/\gamma })$, whence the growth
  of $D_V$ follows. 
\end{proof}

The analysis of the factorization $DD^* = R $ is the key tool to find
all  possible solutions to the phase-retrieval problem for
complex-valued functions in $V_\beta ^\infty (\phi _\gamma )$. In
contrast to the uniqueness among real-valued solutions, there are many
substantially different solutions. In principle, these  can be found by the
following procedure.

\textbf{Reconstruction procedure.}
 Let $\Lambda \subseteq \bR $ with density $D^- (\Lambda )>2$. Let
 $f\in V_1^\infty (\phi _\gamma )$ be arbitrary  and $\{|f(\lambda )|: \lambda \in
 \Lambda \}$ be given.
 To find all functions $h$ that recover the phaseless values
 $|h(\lambda )| = |f(\lambda )|, \lambda \in \Lambda $ we proceed as
 follows:
 
(i) Find the coefficient sequence $\tilde{r} \in \ell ^\infty (\bZ)$
solving the sampling problem 
$$
|f(\lambda ) |^2 = \sum _{n\in \bZ} \widetilde{r_n} e^{-2 \gamma
  (\lambda - n/2)^2} \qquad \lambda \in \Lambda \, ,
$$
for some function $|f|^2 \in V^\infty _{1/2} (\phi _{2\gamma })$. 

(ii) 
Let $Z= \{z_j: j\in \bN\}$ be the zero set    of $R(z) = \sum _{n\in \bZ }
\widetilde{r_n} e^{-\gamma n^2/2}\,  e^{2\pi i n z}$ in the vertical
strip $S = \{x+iy\in \bC :
-1/2 < x \leq 1/2\}$.  Then $R^*=R$  by \eqref{eq:7} and  $R = \Pi (Z, 2m, 2r )$ for
some $m,r \in \bN $ by  Lemma~\ref{hada}.

(iii) Choose $V\subseteq Z\cap S_+$ and define  $D_V$ by
\eqref{fact1}.

(iv) 
Determine the Fourier coefficients of  $D_V(\xi )$: 
$$
d_k = \int _0^1 D_V(\xi ) e^{-2\pi i k \xi } \, d\xi \, .
$$
and set $c_k = d_k e^{\gamma k^2}$ and  
$$
f_V(x) = \alpha \sum  c_k  e^{-\gamma (x-k)^2} \, 
$$
for some $\alpha \in \bC , |\alpha | =1$.
Since $|R(\xi +iy)| = \cO (e^{2\pi ^2 y^2/\gamma })$ by Lemma~\ref{comp},  $|D_V(\xi
  +iy)| = \cO(e^{\pi ^2 y^2/\gamma })$ by Lemma~\ref{summ}. By
  Lemma~\ref{comp} the Fourier coefficients of $D_V$ have Gaussian
  decay and consequently $c_k = d_k e^{\gamma k^2}$ is  bounded.
It follows that  $f\in V^\infty _1 (\phi _\gamma )$. By construction
$|f_V(\lambda )| = |f(\lambda )|$ for all $\lambda \in \Lambda
$. Furthermore, every function $h\in V^\infty _1(\phi _\gamma )$
satisfying $|h(\lambda ) | = |f(\lambda )|, \lambda \in \Lambda ,$
must be  of the form $f_V$. 

\vspace{3mm}

\noindent \textbf{Remarks.} 
1. If $f$ is real-valued, then $D^*=D$,  and its zero set is
symmetric, 
$W=JW$, consequently the zero set $Z= W\dot{\cup } JW$ contains all zeros of
$D$ with double multiplicity.  Since every zero has even multiplicity,
we can set $Z\cap S_+ = V \dot { \cup } V$, where $V$ are the zeros of
$R$ in $S_+$ with half the multiplicity. Equation \eqref{fact0} then
yields $W= V \cup JV \cup Z_0 = JW$. The corresponding function $D_V$
satisfies $D_V= D_V^*$, and $f_V$ is the real-valued solution of the
phase-retrieval problem. We note that other choices of $V$ yield
complex-valued functions $f_V$ such that $|f_V(\lambda ) | =
|f(\lambda )|$.

For real-valued $f$  the steps (i) --- (iv) constitute  a reconstruction procedure
of $f$ from its unsigned samples $|f(\lambda )|,\lambda \in \Lambda
$. Whereas Theorem~\ref{main} only asserts the uniqueness up to a
sign, the factorization of Lemma~\ref{hada} also implies a (rather
theoretical) reconstruction. 

2. If $f\in V_1^\infty (\phi _\gamma )$ is   \emph{real-valued and even}, 
then $c_k = \overline{c_k} =
c_{-k}$. Consequently,   $\hat{d}(\xi ) = \sum _{k\in \bZ} c_k e^{-\gamma k^2}
e^{2\pi i k\xi }= \overline{\hat{d}(-\xi)}$ is real-valued, even, and smooth,  and $\hat{r}(\xi ) = \hat{d}(\xi
)^2\geq 0$. Thus  $\hat{r} \geq 0$, and the only smooth
solutions are  $\hat{d} = \pm \hat{r}^{1/2}$. 
In this case we can obtain  the coefficients $c_k$ directly as the
Fourier coefficients of $\hat{r}^{1/2}$ via
\begin{equation*}
  \label{eq:36}
c_k e^{-\gamma k^2} = \int _{0}^1 \hat{r}(\xi )^{1/2} \, e^{-2\pi i
  k\xi } \, d\xi \, .  
\end{equation*}
Of course, for the boundedness of these
coefficients 
we still need the analysis that led to  Lemma~\ref{summ}.

3. \emph{Stability.} The  procedure in steps (i) ---
(iv) is only of theoretical interest, because 
it  is well-known  that phase-retrieval in infinite
dimensions is inherently unstable~\cite{CCD16,AGr17,ADGT17}.  This is completely obvious
in the reconstruction procedure in the \sis\ space $V^\infty _1(\phi
_\gamma )$: numerically, the transition from the relevant coefficient
sequence $c$ to $d$, given by $d_k = c_k e^{-\gamma k^2}$ amounts to
the truncation of $c$ to a finite sequence. Once the sequence $d$ has
been obtained (steps (iii) and (iv) above), the transition $c_k = d_k
e^{\gamma k^2}$ leads to an amplification of all accumulated
errors. Yet,  despite the inherent instability,   several
 steps in the above reconstruction of $f$ are stable. The relevant
 estimate is \eqref{eq:13}, for the reconstruction of $\tilde{r}$ from the
 phaseless samples $|f(\lambda )|$. For coefficient sequences $c$  with
 small support the reconstruction promises to be reasonably
 effective, which is consistent with the arguments in  \cite{ADGT17}.

  \section{Proof of the Factorization Lemma}
\label{hadaa}
 As we do not know a precise
citation for the statement, we include a  proof of
Lemma~\ref{hada}. We need to show that every periodic  entire function
$D$ of order $2$ (more generally, of finite order) can be factored into a
product with factors of the form
$$
\frac{e^{-2\pi i z} - e^{-2\pi i w}}{1-  e^{-2\pi i w} } \qquad \text{
  or } \qquad   \frac{e^{2\pi i z} - e^{2\pi i w}}{1-  e^{2\pi i w} } \, ,
$$
where $w$ is a zero of $D$ in the vertical strip $S = \{x+iy\in \bC :
-1/2 < x \leq 1/2\}$.  The choice of the signs depends on the sign of
$\mathrm{Im}\, w$. As we will see in part  (vi) of the proof, the sign is determined by the
asymptotic behavior of $\cot \pi w$ as $\mathrm{Im}\, w \to \pm \infty
$.


\begin{proof}
(i) Since $D$ is periodic, its zero set is periodic and, by
definition of $W_\pm $ the zero set is $\bigcup _{w\in W_+ \cup W_-} (w +\bZ )
$. 
Since $D $ is entire of order $2$, the convergence exponent of its
zeros is $>2$~\cite{simon2}. This implies that 
\begin{equation}
  \label{eq:32}
  \sum _{w \in W, w\neq 0} \sum _{k\in \bZ } \frac{1}{|w+k|^{3}}
  <\infty \, .
\end{equation}
(ii) Let
$$
\tilde{D}(z) = \big(e^{2\pi i z} -1\big)^m \,   \, \prod _{w\in W_+}
  \frac{e^{-2\pi i z} - e^{-2\pi i w}}{1-  e^{-2\pi i w} }\,  \prod _{w\in W_-}
  \frac{e^{2\pi i z} - e^{2\pi i w}}{1-  e^{2\pi i w} }
  $$
  be the main part of the right-hand side of \eqref{eq:oc1}. We first
  check the convergence of the product. For this  we 
need to verify that
$$
M_R := \sup _{|z| \leq R} \sum _{w\in W_-} \big| \frac{e^{2\pi i z}-e^{2\pi i
    w}}{1-e^{2\pi i w}} - 1 \big|
$$
is finite for every $R>0$. This expression is simply
\begin{align}
  M_R &=   \sup _{|z| \leq R} |e^{2\pi i z} - 1|  \sum _{w\in W_-}
        \frac{1}{|1-e^{2\pi i w}|} \notag  \\
  &\leq   \sup _{|z| \leq R} |e^{2\pi i z} - 1|  \sum _{w\in W_-}
        \frac{1}{e^{2\pi | \mathrm{Im} \,  w| } -1} \, . \label{eq:38}
\end{align}
Since for $w\in W_-$, $e^{-\mathrm{Im}\, w} \geq
1+|\mathrm{Im}\, w|^3/6$, the sum in \eqref{eq:38} converges by
\eqref{eq:32}.  Thus $\prod _{w\in W_-} $ converges uniformly on
compact sets. By a similar argument the product $\prod _{w\in W_+}   \frac{e^{-2\pi
    i z} - e^{-2\pi i w}}{1-  e^{-2\pi i w} }  $ converges 
uniformly on compact sets.
Consequently
$\tilde{D}$ is an entire function.  Clearly $\tilde{ D}$ is periodic with period $1$, and its zero set in the
strip $S$ is precisely $W_- \cup W_+ $ ($\cup \{0\}$, if $0$ is a zero). By construction, $D$ and
$\tilde{ D}$ possess the same zero set with the same multiplicities. 

(iii) Next  we use the Hadamard factorization theorem to factorize $D$
with respect to its zeros as follows: 
$$
D(z) = \Big( z \prod _{k\in \bZ, k\neq 0} \big(1-\frac{z}{k}\big)
e^{\frac{z}{k}+\frac{z^2}{2k^2}}\Big)^m \, \prod _{w\in W} \prod _{k\in \bZ} \big(1-\frac{z}{w+k}\big)
e^{\frac{z}{w+k}+\frac{z^2}{2(w+k)^2}} \, e^{p(z)} , 
$$
where $m$ is the order of the zero at $0$ and $p$ is a quadratic polynomial.

We note right away that the first factor is a power of 
$$
G(z,0) :=  z \prod _{k\in \bZ, k\neq 0} \big(1-\frac{z}{k}\big)
e^{\frac{z}{k}+\frac{z^2}{2k^2}} = \pi \inv \sin \pi z   \, e^{\pi ^2
  z^2 / 6}
$$
by the factorization of the sine-function. See \cite{simon2}, Section
9.2, or \cite{levin}, Lecture~4.  

(iv) For the  factors with $w\neq 0$ we introduce 
\begin{equation}
  \label{eq:48}
G(z,w) = \prod _{k\in \bZ} \big(1-\frac{z}{w+k}\big)\, 
\exp \Big(\frac{z}{w+k}+\frac{z^2}{2(w+k)^2}\Big)   \, .
\end{equation}
 Then $G(z,w)$ has simple zeros at $w+\bZ $, as does the
function $\sin (\pi (z-w))$. Using the identities, 
\begin{align}
  \label{eq:39}
  \sum _{k\in \bZ } \frac{1}{w+k} & = \lim _{n\to \infty } \sum _{|k|\leq n}
                                    \frac{1}{w+k} = \pi \cot \pi w  \\
  \sum _{k\in \bZ } \frac{1}{(w+k)^2} & =  \frac{\pi ^2}{\sin ^2 \pi
                                        w} \, , \label{eq:40}
\end{align}
for  $w\neq 0$, e.g.,  \cite{conway78,simon2},  we identify $G(z,w)$ as
\begin{equation}
  \label{eq:41}
  G(z,w) = \frac{\sin \pi (z-w)}{\sin (-\pi w)} \, \exp \Big(\pi z \cot \pi
    w   + \frac{\pi ^2 z^2}{2\sin ^2 \pi w} \Big)
  \, .  
\end{equation}
This is easily seen by  calculating $G(z,w) \frac{\sin (-\pi w)}{\sin
  \pi (z-w)}$ with the  factorization of $\sin \pi z$. The quadratic
polynomial in the exponent is obtained by substituting 
 \eqref{eq:39} and \eqref{eq:40} in \eqref{eq:48}.

(v) Now consider the ratio of corresponding factors in $D$ and
$\tilde{ D}$ and simplify the expression. We argue only 
for   $w\in W_+$. 
\begin{align*}
  \frac{G(z,w) (1-e^{-2\pi i w})}{e^{-2\pi i z}-e^{-2\pi i w}} &=
 \frac{\sin \pi (z-w)}{\sin   (-\pi  w)}  \, \frac{e^{i\pi w}-e^{-i\pi
w}}{e^{-i\pi (z-w)} - e^{i\pi (z-w)}} \, \frac{e^{-i\pi w}}{e^{-i\pi
 z} e^{-i\pi w}} \,   e^{\pi z \cot \pi
    w} \, e^{\frac{\pi ^2 z^2}{2\sin ^2 \pi w}} \\
&= e^{\pi z (\cot \pi w +i)} e^{\pi ^2  z^2 /(2\sin ^2 \pi w)} \, ,
\end{align*}
whereas for $w\in W_-$ we have
$$
\frac{G(z,w) (1-e^{2\pi i w})}{e^{2\pi i z}-e^{2\pi i w}} 
= e^{\pi z (\cot \pi w - i)} e^{\pi ^2  z^2 /(2\sin ^2 \pi w)} \, .
$$
Next we take the product over $w\in
W_-$ and assume  for now that the product converges. Then  we obtain
\begin{equation}
  \label{eq:45}
\prod _{w\in W_+}   \frac{G(z,w) (1-e^{-2\pi i w})}{e^{-2\pi i
    z}-e^{-2\pi i w}} \exp \Big( \pi z \sum _{w\in W_+} (\cot \pi w
+i) + \pi ^2z^2 \sum _{w\in W_+} \frac{1}{2 \sin ^2 \pi w}\Big) =
e^{q_+(z)}  
\end{equation}
for some quadratic polynomial $q_+$. Likewise
$$\prod _{w\in W_-}   \frac{G(z,w) (1-e^{2\pi i w})}{e^{2\pi i z}-e^{2\pi i w}} =
 e^{q_-(z)}
$$
for a  quadratic polynomial $q_-$, and 
$G(z,0)/   \big(e^{2\pi i z} -1\big)= \pi \inv \exp \Big( -i\pi z + \pi ^2z^2/2
\Big) = e^{q_0(z)}$.  We therefore  obtain that 
\begin{equation}
  \label{eq:42}
  \frac{D(z)}{\tilde{D}(z)} = e^{p(z)+ q_+(z)+q_-(z)+q_0(z)} \, .
\end{equation}
The left-hand side is an entire function with period $1$, whereas the right-hand
side is the exponential of a quadratic polynomial. It is now easy to
see that   $e^P$ is periodic for a polynomial $P$ , \fif\ $P(z) = 2\pi i
r z$ for some $r\in \bZ $. We conclude that $D(z) = \tilde{D}(z)
e^{2\pi i rz}$, and this is precisely 
\eqref{eq:oc1}. 

(vi) It remains to be shown that the  product in \eqref{eq:45}
converges. This follows from 
\begin{equation*}
  \label{eq:46}
\sum _{w\in W_+}  |\cot \pi w + i| = \sum _{w\in W_+}
\Big|i \, \frac{e^{i\pi w}+e^{-i\pi w}}{e^{i\pi w}-e^{-i\pi w}} +
i\Big| = 2 \sum _{w\in W_+} \frac{|e^{i\pi w}|}{|e^{i\pi w}- e^{-i\pi w}|}
 \lesssim  \sum _{w\in W_+} e^{-2\pi \mathrm{Im} w} \, .
\end{equation*}
Since this sum is over the terms with $\mathrm{Im}w >0$ and $\sum
_{w\in W} |w|^{-3} <\infty $, it is finite. For the terms in $W_-$ our
choice of the signs yields the same conclusion.  Finally
$\sum _{w\in W} |\frac{1}{2 \sin ^2 \pi w}|$ is finite by the same
reason. 
\end{proof}

\vspace{2mm}
\textbf{Acknowledgement:} The author  would like to thank the Isaac
Newton Institute for Mathematical Sciences, Cambridge, for support and
hospitality during the programme ``Approximation, sampling and
compression in data science'' where work on 
this paper was undertaken. Special thanks go to Felix Voigtl\"ander
for his critical reading and his  detailed comments that led me to add
Section~\ref{complex}.



 \def\cprime{$'$} \def\cprime{$'$} \def\cprime{$'$} \def\cprime{$'$}
  \def\cprime{$'$} \def\cprime{$'$}


\end{document}